\documentclass[twoside,11pt]{article}

\usepackage{jmlr2e}

\usepackage{natbib}

\usepackage{algorithm}
\usepackage{algorithmic}

\usepackage{amssymb}
\usepackage{amsmath}
\usepackage{multirow}
\usepackage{mathrsfs}
\usepackage{bbm}
\usepackage{bm}
\usepackage{eufrak}
\usepackage{bookmark}
\usepackage{mathtools}

\ShortHeadings{On the Suboptimality of Proximal Gradient Descent for $\ell^{0}$ Sparse Approximation}{Y. Yang et al.}
\firstpageno{1}

\newcommand{\cN}{\mathcal{N}}
\newcommand{\cO}{\mathcal{O}}

\newcommand{\bg}{\mathbf{g}}
\newcommand{\bv}{\mathbf{v}}
\newcommand{\bx}{\mathbf{x}}

\newcommand{\bu}{\mathbf{u}}
\newcommand{\by}{\mathbf{y}}
\newcommand{\bz}{\mathbf{z}}

\newcommand{\bA}{\mathbf{A}}
\newcommand{\bB}{\mathbf{B}}
\newcommand{\bD}{\mathbf{D}}

\newcommand{\bF}{\mathbf{F}}
\newcommand{\bG}{\mathbf{G}}
\newcommand{\bR}{\mathbf{R}}
\newcommand{\bT}{\mathbf{T}}
\newcommand{\bS}{\mathbf{S}}

\newcommand{\bP}{{\mathbf{P}}}

\newcommand{\bH}{{\mathbf{H}}}
\newcommand{\bI}{{\mathbf{I}}}

\newcommand{\bW}{{\mathbf{W}}}

\newcommand{\bQ}{{\mathbf{Q}}}

\newcommand{\bbeta}{\bm{\beta}}

\newcommand{\N}{{\rm I}\kern-0.18em{\rm N}}

\newcommand{\R}{{\rm I}\kern-0.18em{\rm R}}
\newcommand{\h}{{\rm I}\kern-0.18em{\rm H}}
\newcommand{\K}{{\rm I}\kern-0.18em{\rm K}}
\newcommand{\p}{{\rm I}\kern-0.18em{\rm P}}
\newcommand{\E}{{\rm I}\kern-0.18em{\rm E}}
\newcommand{\Z}{{\rm Z}\kern-0.18em{\rm Z}}

\newcommand{\1}{{\rm 1}\kern-0.25em{\rm I}}

\newcommand{\pn}{\p_{\kern-0.25em n}}
\newcommand{\pnm}{\p_{\kern-0.25em n,m}}
\newcommand{\psubm}{\p_{\kern-0.25em m}}

\newcommand{\BigO}[1]{{\operatorname{O}}}

\DeclareMathOperator*{\argmin}{arg\,min}

\newtheorem{MyDefinition}{Definition}
\newtheorem{MyLemma}{Lemma}
\newtheorem{MyTheorem}{Theorem}

\newtheorem{MyRemark}{Remark}

\begin{document}

\title{On the Suboptimality of Proximal Gradient Descent for $\ell^{0}$ Sparse Approximation}

\author{\name Yingzhen Yang \email superyyzg@gmail.com \\
       \addr Snap Research, USA\\
       \AND
       \name Jiashi Feng \email elefjia@nus.edu.sg \\
       \addr Department of ECE, National University of Singapore, Singapore\\
       \AND
       Nebojsa Jojic \email jojic@microsoft.com \\
       Microsoft Research, USA\\
       \AND
       Jianchao Yang \email jianchao.yang@snap.com \\
       \addr Snap Research, USA\\
       \AND
       Thomas S. Huang \email t-huang1@illinois.edu \\
       \addr Beckman Institute, University of Illinois at Urbana-Champaign, USA\\
       }

\editor{}

\maketitle

\begin{abstract}
We study the proximal gradient descent (PGD) method for $\ell^{0}$ sparse approximation problem as well as its accelerated optimization with randomized algorithms in this paper. We first offer theoretical analysis of PGD showing the bounded gap between the sub-optimal solution by PGD and the globally optimal solution for the $\ell^{0}$ sparse approximation problem under conditions weaker than Restricted Isometry Property widely used in compressive sensing literature. Moreover, we propose randomized algorithms to accelerate the optimization by PGD using randomized low rank matrix approximation (PGD-RMA) and randomized dimension reduction (PGD-RDR). Our randomized algorithms substantially reduces the computation cost of the original PGD for the $\ell^{0}$ sparse approximation problem, and the resultant sub-optimal solution still enjoys provable suboptimality, namely, the sub-optimal solution to the reduced problem still has bounded gap to the globally optimal solution to the original problem.
\end{abstract} 
\section{Introduction}
In this paper, we consider the $\ell^{0}$ sparse approximation problem, also named the $\ell^{0}$ penalized Least Square Estimation (LSE) problem
\begin{small}\begin{align}\label{eq:l0-sparse-appro}
&\mathop {\min }\limits_{{\bz \in \R^n} } L(\bz) = {\|\bx - \bD \bz\|_2^2 + {\lambda}\|{\bz}\|_0}
\end{align}\end{small}%
where $\bx \in \R^d$ is a signal in $n$-dimensional Euclidean space, $\bD$ is the design matrix of dimension $d \times n$ which is also called a dictionary with $n$ atoms in the sparse coding literature. The goal of problem (\ref{eq:l0-sparse-appro}) is to approximately represent signal $\bx$ by the atoms of the dictionary $\bD$ while requiring the representation $\bz$ to be sparse. Due to the nonconvexity imposed by the $\ell^{0}$-norm, extensive existing works resort to solve its $\ell^{1}$ relaxation
\begin{small}\begin{align}\label{eq:l1-sparse-appro}
&\mathop {\min }\limits_{{\bz \in \R^n} } {\|\bx - \bD \bz\|_2^2 + {\lambda}\|{\bz}\|_1}
\end{align}\end{small}%
(\ref{eq:l1-sparse-appro}) is convex and also known as Basis Pursuit Denoising which can be solved efficiently by linear programming or iterative shrinkage algorithms \citep{Daubechies-iter-shrink-sparse04,Elad-shrinkage06,Bredies-hard-iter-shrink08}. Albeit the nonconvexity of (\ref{eq:l0-sparse-appro}), sparse coding methods such as \citep{Mancera2006,BaoJQS14} that directly optimize virtually the same objective as (\ref{eq:l0-sparse-appro}) demonstrate compelling performance compared to its $\ell^{1}$ norm counterpart in various application domains such as data mining, applied machine learning and computer vision. Cardinality constraint in terms of $\ell^{0}$-norm is also studied for M-estimation problems \citep{Jain14-IHT-Mestimation}.

We use the proximal gradient descent (PGD) method to obtain a sub-optimal solution to (\ref{eq:l0-sparse-appro}) in an iterative shrinkage manner with theoretical guarantee. Although the Iterative Hard-Thresholding (IHT) algorithm proposed by \citep{Blumensath-hard-iter-shrink2008} also features iterative shrinkage, we prove the bound for gap between the sub-optimal solution and the globally optimal solution to (\ref{eq:l0-sparse-appro}). Our result of the bounded gap only requires nonsingularity of the submatrix of $\bD$ with the columns in the support of the sub-optimal and globally optimal solution (see the subsection ``Assumptions for Our Analysis'' for details). To the best of our knowledge, this is the first result analyzing the gap between sub-optimal solution and the globally optimal solution for the important $\ell^{0}$ sparse approximation problem under assumptions weaker than Restricted Isometry Property (RIP) \citep{Candes2008}. The most related research is presented in in \citep{zhang2012}  where the distance between two local solutions of the concave regularized problems under much more restrictive assumptions including sparse eigenvalues of the dictionary. Moreover, our results suggest the merit of sparse initialization.

Furthermore, we propose to accelerate PGD for the $\ell^{0}$ sparse approximation problem by two randomized algorithms. We propose proximal gradient descent via Randomized Matrix Approximation (PGD-RMA) which employs rank-$k$ approximation of the dictionary via random projection. PGD-RMA reduces the cost of computing the gradient during the gradient descent step of PGD from $\cO(dn)$ to $\cO(dk + nk)$ by solving the reduced problem instead of the original problem with $k \ll \min\{d,n\}$, for a dictionary of size $d \times n$. The second randomized algorithm is proximal gradient descent via Randomized Dimension Reduction (PGD-RDR) which employs random projection to generate dimensionality-reduced signal and dictionary. PGD-RDR reduces the computational cost of the gradient descent step of PGD from $\cO(dn)$ to $\cO(mn)$ with $m < d$. While previous research focuses on the theoretical guarantee for convex problems via such randomized optimization \citep{Drineas2011-fast-least-square-approximation,ZhangWZ2016-sparse-linear-random-projection}, we present the gap between the sub-optimal solution to the reduced problem and the globally optimal solution to the original problem. Our result establishes provable and efficient optimization by randomized low rank matrix decomposition and randomized dimension reduction for the nonconvex and nonsmooth $\ell^{0}$ sparse approximation problem, while very few results are available in the literature in this direction.

\subsection*{Notations}
Throughout this paper, we use bold letters for matrices and vectors, regular lower letter for scalars. The bold letter with subscript indicates the corresponding element of a matrix or vector, and $\|\cdot\|_p$ denote the $\ell^{p}$-norm of a vector, or the $p$-norm of a matrix. We let $\bbeta_{\bI}$ denote the vector formed by the elements of $\bbeta$ with indices in $\bI$ when $\bbeta$ is a vector, or matrix formed by columns of $\bbeta$ with indices being the nonzero elements of $\bI$ when $\bbeta$ is a matrix. ${\rm supp}(\cdot)$ indicates the support of a vector, i.e. the set of indices of nonzero elements of this vector. $\sigma_{\min}(\cdot)$ and $\sigma_{\max}(\cdot)$ indicate the smallest and largest nonzero singular value of a matrix. 
\section{Proximal Gradient Descent for $\ell^{0}$ Sparse Approximation}

Solving the $\ell^{0}$ sparse approximation problem (\ref{eq:l0-sparse-appro}) is NP-hard in general \citep{Natarajan95-sparse-linear-system}. Therefore, the literature extensively resorts to approximate algorithms, such as Orthogonal Matching Pursuit \citep{Tropp04}, or that using surrogate functions \citep{Hyder09}, for $\ell^{0}$ problems. In addition, PGD has been used by \citep{BaoJQS14} to find the approximate solution to (\ref{eq:l0-sparse-appro}) with sublinear convergence to the critical point of the objective of (\ref{eq:l0-sparse-appro}), as well as satisfactory empirical results. The success of PGD raises the interesting question that how good the approximate solution by PGD is.

In this section, we first present the algorithm that employs PGD to optimize (\ref{eq:l0-sparse-appro}) in an iterative shrinkage manner. Then we show the suboptimality of the sub-optimal solution by PGD in terms of the gap between the sub-optimal solution and the globally optimal solution.
\subsection{Algorithm}
In $t$-th iteration of PGD for $t \ge 1$, gradient descent is performed on the squared loss term of $L(\bz)$, i.e. $Q(\bz) \triangleq \|\bx - \bD {\bz}\|_2^2$, to obtain
\begin{small}\begin{align}\label{eq:l0-sa-proximal-step1}
\tilde {\bz}^{(t)} = {\bz}^{(t-1)} - \frac{2}{{\tau}s} ({\bD^\top}{\bD}{{\bz}^{(t-1)}}-{\bD^\top}{\bx})
\end{align}\end{small}%
where $\tau$ is any constant that is greater than $1$. $s>0$ is usually chosen as the Lipschitz constant for the gradient of function $Q(\cdot)$, namely
\begin{small}\begin{align}\label{eq:lipschitz-Q}
\|\nabla Q(\by) - \nabla Q(\bz)\|_2 \le s \|\by-\bz\|_2, \,\, \forall \, \by,\bz \in \R^n
\end{align}\end{small}%
${\bz}^{(t)}$ is then the solution to the following the proximal mapping:
\begin{small}\begin{align}\label{eq:l0-sa-subprob}
&{\bz}^{(t)} = \argmin \limits_{\bv \in \R^{n}} {\frac{{\tau} s}{2}\|\bv - {\tilde {\bz}^{(t)}}\|_2^2 + {\lambda}\|\bv\|_0}
\end{align}\end{small}%
and (\ref{eq:l0-sa-subprob}) admits the closed-form solution:
\begin{small}\begin{align}\label{eq:l0-sa-proximal-step2}
&{\bz}^{(t)} = h_{\sqrt{\frac{2\lambda}{{\tau}s}}}(\tilde {\bz}^{(t)})
\end{align}\end{small}%
where $h_{\theta}$ is an element-wise hard thresholding operator:
\begin{small}\begin{align*}
[h_{\theta}(\bu)]_j=
\left\{
\begin{array}
    {r@{\quad:\quad}l}
    0 & {|\bu_j| < \theta } \\
    {\bu_j} & {\rm otherwise}
\end{array}
\right., \quad 1 \le j \le n
\end{align*}\end{small}%
The iterations start from $t=1$ and continue until the sequence $\{L({\bz}^{(t)})\}_t$ or $\{{\bz}^{(t)}\}_t$ converges or maximum iteration number is achieved. The optimization algorithm for the $\ell^{0}$ sparse approximation problem (\ref{eq:l0-sparse-appro}) by PGD is described in Algorithm~\ref{alg:pgd}. In practice, the time complexity of optimization by PGD is $\cO(Mdn)$ where $M$ is the number of iterations (or maximum number of iterations) for PGD.
\begin{algorithm}[!ht]
\renewcommand{\algorithmicrequire}{\textbf{Input:}}
\renewcommand\algorithmicensure {\textbf{Output:} }
\caption{Proximal Gradient Descent for the $\ell^{0}$ Sparse Approximation (\ref{eq:l0-sparse-appro})  }
\label{alg:pgd}
\begin{algorithmic}[1]
\REQUIRE ~~\\
The given signal $\bx \in R^d$, the dictionary $\bD$, the parameter $\lambda$ for the weight of the $\ell^{0}$-norm, maximum iteration number $M$, stopping threshold $\varepsilon$, the initialization ${\bz}^{(0)} \in \R^n$.
\STATE{Obtain the sub-optimal solution ${{\tilde \bz}}$ by the proximal gradient descent (PGD) method with (\ref{eq:l0-sa-proximal-step1}) and (\ref{eq:l0-sa-proximal-step2})} starting from $t = 1$. The iteration terminates either $\{{\bz}^{(t)}\}_t$ or $\{L({\bz}^{(t)})\}_t$ converges under certain threshold or the maximum iteration number is achieved.
\ENSURE Obtain the sparse code $\hat \bz$ upon the termination of the iterations.
\end{algorithmic}
\end{algorithm}
\subsection{Theoretical Analysis}
In this section we present the bound for the gap between the sub-optimal solution by PGD in Algorithm~\ref{alg:pgd} and the globally optimal solution for the $\ell^{0}$ sparse approximation problem (\ref{eq:l0-sparse-appro}). With proper initialization $\bz^{(0)}$, we show that the sub-optimal solution by PGD is actually a critical point of $L(\bz)$ in Lemma~\ref{lemma::PGD-convergence}, namely the sequence $\{{\bz}^{(t)}\}_t$ converges to a critical point of the objective (\ref{eq:l0-sparse-appro}). We then show that both this sub-optimal solution and the globally optimal solution to (\ref{eq:l0-sparse-appro}) are local solutions of a carefully designed capped-$\ell^{1}$ regularized problem in Lemma~\ref{lemma::equivalence-to-capped-l1}. The bound for $\ell^{2}$-distance between the sub-optimal solution and the globally optimal solution is then presented in Theorem~\ref{theorem::suboptimal-optimal}. In the following analysis, we let $\bS = {\rm supp}({\bz}^{(0)})$. Also, since $\ell^{0}$ is invariant to scaling, the original problem (\ref{eq:l0-sparse-appro}) is equivalent to $\mathop {\min }\limits_{{\bz \in \R^n} } {\|\bx - \frac{\bD}{m} \bz^{'}\|_2^2 + {\lambda}\|{\bz^{'}}\|_0}$ for $\bz^{'} = m \bz$ with $m > \max_i \|\bD^i\|_2$, so it is safe to assume $\max_i \|\bD^i\|_2 \le 1$. Without loss of generality, we let $\|\bx\|_2 \le 1$.

\begin{MyLemma}\label{lemma::shrinkage-sufficient-decrease}
(\textit{Support shrinkage and sufficient decrease of the objective function})
Choose ${\bz}^{(0)}$ such that  $\|\bx - \bD {\bz}^{(0)}\|_2 \le 1$. When $s > \max\{2|\bS|, \frac{2(1+{\lambda}|\bS|)}{\lambda \tau}\}$, then
\begin{small}\begin{align}\label{eq:support-shrinkage}
{\rm supp} ({\bz}^{(t)}) \subseteq {\rm supp} ({\bz}^{(t-1)}), t \ge 1
\end{align}\end{small}%
namely the support of the sequence $\{{\bz}^{(t)}\}_t$ shrinks. Moreover, the sequence of the objective $\{L({\bz}^{(t)})\}_t$  decreases, and the following inequality holds for $t \ge 1$:
\begin{small}\begin{align}\label{eq:l0graph-proximal-sufficient-decrease}
&L({\bz}^{(t)}) \le L({\bz}^{(t-1)}) - \frac{(\tau-1)s}{2} \|{\bz}^{(t)} - {\bz}^{(t-1)}\|_2^2
\end{align}\end{small}%
And it follows that the sequence $\{L({\bz}^{(t)})\}_t$ converges.
\end{MyLemma}
\begin{MyRemark}
One can always choose ${\bz}^{(0)}$ as the optimal solution to the $\ell^{1}$ regularized problem, so $\|\bx - \bD {\bz}^{(0)}\|_2^2 + {\lambda}\|{\bz}^{(0)}\|_1 \le \|\bx\|_2^2 \le 1$, and it follows that $\|\bx - \bD {\bz}^{(0)}\|_2^2 \le 1$. Also, when not every subspace spanned by linearly independent columns of $\bD$ is orthogonal to $\bx$ (which is common in practice), we can always find ${\bz}^{(0)}$ such that  $\|\bx - \bD {\bz}^{(0)}\|_2^2 \le 1$ and ${\bz}^{(0)}$ is a nonzero vector. The support shrinkage property by Lemma~\ref{lemma::shrinkage-sufficient-decrease} shows that the original $\ell^{0}$ sparse approximation (\ref{eq:l0-sparse-appro}) is equivalent to a dictionary-reduced version
\begin{align}\label{eq:l0-sparse-appro-reduced}
&\mathop {\min }\limits_{{\bz \in \R^n} } L(\bz) = {\|\bx - \bD_{\bS} \bz\|_2^2 + {\lambda}\|{\bz}\|_0}
\end{align}
since PGD would not choose dictionary atoms outside of $\bD_{\bS}$, so the numerical computation of PGD can be improved by using $\bD_{\bS}$ as the dictionary with $s > \max\{2|\bS|, \frac{2(1+{\lambda}|\bS|)}{\lambda \tau}\}$. When $\bz^{(0)}$ is sparse, $\bD_{\bS}$ has a small number of atoms which enables fast optimization of the reduced problem (\ref{eq:l0-sparse-appro-reduced}). Also note that $s$ can be much smaller than the Lipschitz constant for the gradient of function $Q(\cdot)$ by the choice indicated by Lemma~\ref{lemma::shrinkage-sufficient-decrease}, which leads to a larger step size for the gradient descent step (\ref{eq:l0-sa-proximal-step1}).
\end{MyRemark}

The definition of critical points are defined below which is important for our analysis.

\begin{MyDefinition}
(\textit{Critical points})
Given the non-convex function $f \colon \R^n \to R \cup \{+\infty\}$ which is a proper and lower semi-continuous function.
\begin{itemize}
\item for a given $\bx \in {\rm dom}f$, its Frechet subdifferential of $f$ at $\bx$, denoted by $\tilde \partial f(x)$, is
the set of all vectors $\bu \in \R^n$ which satisfy
\begin{small}\begin{align*}
&\limsup\limits_{\by \neq \bx,\by \to \bx} \frac{f(\by)-f(\bx)-\langle \bu, \by-\bx \rangle}{\|\by-\bx\|} \ge 0
\end{align*}\end{small}%
\item The limiting-subdifferential of $f$ at $\bx \in \R^n$, denoted by written $\partial f(x)$, is defined by
\begin{small}\begin{align*}
&\partial f(x) = \{\bu \in \R^n \colon \exists \bx^k \to \bx, f(\bx^k) \to f(\bx), \tilde \bu^k \in {\tilde \partial f}(\bx^k) \to \bu\}
\end{align*}\end{small}%
\end{itemize}
The point $\bx$ is a critical point of $f$ if $0 \in \partial f(x)$.
\end{MyDefinition}

If $\bD_{\bS}$ is nonsingular, Lemma~\ref{lemma::PGD-convergence} shows that the sequences $\{{\bz}^{(t)}\}_t$ produced by PGD converges to a critical point of $L(\bz)$, the objective of the $\ell^{0}$ sparse approximation problem (\ref{eq:l0-sparse-appro}).

\begin{MyLemma}\label{lemma::PGD-convergence}
With $\bz^{(0)}$ and $s$ in Lemma~\ref{lemma::shrinkage-sufficient-decrease}, if $\bD_{\bS}$ is nonsingular, then the sequence $\{{\bz}^{(t)}\}_t$ generated by PDG with (\ref{eq:l0-sa-proximal-step1}) and (\ref{eq:l0-sa-proximal-step2}) converges to a critical point of $L(\bz)$.
\end{MyLemma}
Denote the critical point of $L(\bz)$ by ${\hat \bz}$ that the sequence $\{{\bz}^{(t)}\}_t$ converges to when the assumption of Lemma~\ref{lemma::PGD-convergence} holds, and denote by ${\bz}^*$ the globally optimal solution to the $\ell^{0}$ sparse approximation problem (\ref{eq:l0-sparse-appro}).
Also, we consider the following capped-$\ell^{1}$ regularized problem, which replaces the noncontinuous $\ell^{0}$-norm with the continuous capped-$\ell^{1}$ regularization term $R$:
\begin{small}\begin{align}\label{eq:capped-l1-problem}
&\mathop {\min }\limits_{{\bbeta \in \R^n}} L_{{\rm capped-}\ell^{1}}(\bbeta) = \|\bx - \bD \bbeta\|_2^2 + \bR(\bbeta;b)
\end{align}\end{small}%
where $\bR(\bbeta;b) = \sum\limits_{j=1}^n R(\bbeta_j;b)$, $R(t;b) = {\lambda}\frac{\min\{|t|,b\}}{b}$ for some $b > 0$. It can be seen that $R(t;b)$ approaches the $\ell^{0}$-norm when $b \to 0+$. Our following theoretical analysis aims to obtain the gap between ${\hat \bz}$ and ${\bz}^*$. For the sake of this purpose, the definition of local solution and degree of nonconvexity of a regularizer are necessary and presented below.
\begin{MyDefinition}
(\textit{Local solution})
A vector $\tilde \bbeta$ is a local solution to the problem (\ref{eq:capped-l1-problem}) if
\begin{small}\begin{align}\label{eq:cond-local-solution}
&\| 2{\bD^{\top}}({\bD} {\tilde \bbeta} - \bx ) + {\dot \bR} (\tilde \bbeta;b)\|_2  = 0
\end{align}\end{small}%
where ${\dot \bR(\tilde \bbeta;b) = [\dot R(\tilde \bbeta_1;b),\dot R(\tilde \bbeta_2;b),\ldots,\dot R(\tilde \bbeta_n;b) ]^{\top}}$.
\end{MyDefinition}

Note that in the above definition and the following text, $\dot R(t;b)$ can be chosen as any value between the right differential $\frac{\partial R}{\partial t}(t+;b)$ (or ${\dot R(t+;b)}$) and left differential $\frac{\partial R}{\partial t}(t-;b)$ (or ${\dot R(t-;b)}$).

\begin{MyDefinition}\label{def::degree-nonconvexity}
(\textit{Degree of nonconvexity of a regularizer})
For $\kappa \geq 0$ and $t \in \R$, define
  \[
  \theta(t,\kappa):= \sup_s \{ -{\rm sgn}(s-t) ({\dot P}(s;b) - {\dot P}(t;b)) - \kappa |s-t|\}
  \]
  as the degree of nonconvexity for function $P$.
  If $\bu =(u_1,\ldots,u_n)^\top\in \R^n$, $\theta(\bu,\kappa)=[\theta(u_1,\kappa),\ldots,\theta(u_n,\kappa)]$. $sgn$ is a sign function.
\end{MyDefinition}
Note that $\theta(t,\kappa) = 0$ if $P$ is a convex function.

Let ${\hat \bS} = {\rm supp}( {\hat \bz})$, ${\bS^*} = {\rm supp}( {\bz}^*)$, the following lemma shows that both $ {\hat \bz}$ and ${\bz}^*$ are local solutions to the capped-$\ell^{1}$ regularized problem (\ref{eq:capped-l1-problem}).
\begin{MyLemma}\label{lemma::equivalence-to-capped-l1}
With $\bz^{(0)}$ and $s$ in Lemma~\ref{lemma::shrinkage-sufficient-decrease}, if
\begin{small}\begin{align}\label{eq:b-cond}
&0< b < \min\{\min_{j \in {\hat \bS}} | {\hat \bz}_j|, \frac{\lambda}{ \max_{j \notin {\hat \bS}} |\frac{\partial Q}{\partial {\bz_j}}|_{\bz =  {\hat \bz}}|},
\min_{j \in {\bS^*}} | \bz_j^*|, \frac{\lambda}{ \max_{j \notin {\bS^*}} |\frac{\partial Q}{\partial {\bz_j}}|_{\bz = {\bz}^*}|}  \}
\end{align}\end{small}%
(if the denominator is $0$, $\frac{\lambda}{0}$ is defined to be $+\infty$ in the above inequality),
then both $ {\hat \bz}$ and ${\bz}^*$ are local solutions to the capped-$\ell^{1}$ regularized problem (\ref{eq:capped-l1-problem}).
\end{MyLemma}

\begin{MyTheorem}\label{theorem::suboptimal-optimal}
(\textit{Sub-optimal solution is close to the globally optimal solution})
With $\bz^{(0)}$ and $s$ in Lemma~\ref{lemma::shrinkage-sufficient-decrease}, and suppose $\bD_{\hat \bS \cup \bS^*}$ is not singular with $\kappa_0 \triangleq \sigma_{\min}(\bD_{\hat \bS \cup \bS^*}) > 0$. When $\kappa_0^2 > \kappa > 0$ and $b$ is chosen according to (\ref{eq:b-cond}) as in Lemma~\ref{lemma::equivalence-to-capped-l1} ,let $\bF = ({\hat \bS} \setminus \bS^*) \cup (\bS^* \setminus {\hat \bS})$ be the symmetric difference between $\hat \bS$ and $\bS^*$, then
\begin{small}\begin{align}\label{eq:suboptimal-optimal}
&\|{\hat \bz} - \bz^*\|_2 \le \frac{\big(\sum\limits_{j \in \bF \cap \hat \bS} (\max\{0,\frac{\lambda}{b} - {\kappa} |{\hat \bz}_j - b|\})^2 +
 \sum\limits_{j \in \bF \setminus \hat \bS} (\max\{0,\frac{\lambda}{b} - {\kappa} b\})^2 \big)^{\frac{1}{2}}}{2\kappa_0^2-\kappa}
\end{align}\end{small}%
\end{MyTheorem}

It is worthwhile to connect the assumption on the dictionary  in Theorem~\ref{theorem::suboptimal-optimal} to the Restricted Isometry Property (RIP) \citep{CandesTao05} used frequently in the compressive sensing literature. Before the discussion, the definition of sparse eigenvalues is defined below.
\begin{MyDefinition}
(\textit{Sparse eigenvalues})
The lower and upper sparse eigenvalues of a matrix $\bA$ are defined as
\begin{small}\begin{align*}
& \kappa_-(m) := \min_{\|\bu\|_0 \leq m; \|\bu\|_2=1} \|\bA \bu\|_2^2 \quad \kappa_+(m) := \max_{\|\bu\|_0 \leq m,\|\bu\|_2=1}\|\bA \bu\|_2^2
\end{align*}\end{small}%
\end{MyDefinition}

\subsection*{Assumptions for Our Analysis}
Typical RIP requires bounds such as $\delta_\tau+\delta_{2\tau}+\delta_{3\tau}< 1$ or $\delta_{2\tau} < \sqrt{2}-1$ \citep{Candes2008} for stably recovering the signal from measurements and $\tau$ is the sparsity of the signal, where $\delta_{\tau}=\max\{\kappa_+(\tau)-1,1-\kappa_-(\tau)\}$. It should be emphasized that our bound (\ref{eq:suboptimal-optimal}) only requires nonsingularity of the submatrix of $\bD$ with the columns in the support of $\hat \bz$ and $\bz^*$, which are more general than RIP in the sense of not requiring bounds in terms of $\delta$. To see this point, choose the initialization $\bz^{0}$ such that $|{\rm supp}(\bz^{0})| \le \tau$, then the condition $\delta_{2\tau} < \sqrt{2}-1$ \citep{Candes2008} indicates that $\sigma_{\min}(\bD_{\hat \bS \cup \bS^*}) > 2 - \sqrt{2}$, which is stronger than our assumption that $\sigma_{\min}(\bD_{\hat \bS \cup \bS^*}) > 0$. In addition, our assumption on the nonsingularity of $\bD_{\hat \bS \cup \bS^*}$ is much weaker than that of the sparse eigenvalues used in RIP conditions which require minimum eigenvalue of every submatrix of $\bD$ of specified number of columns. 
\begin{MyRemark}\label{remark::suboptimal-optimal}
If $\bz^{(0)}$ is sparse, $ {\hat \bz}$ is also sparse by the property of support shrinkage in Lemma~\ref{lemma::shrinkage-sufficient-decrease}. We can then expect that $|\hat \bS \cup \bS^*|$ is reasonably small, and a small $|\hat \bS \cup \bS^*|$ often increases the chance of a larger $\sigma_{\min}(\bD_{|\hat \bS \cup \bS^*|})$. Also note that the bound for distance between the sub-optimal solution and the globally optimal solution presented in Theorem~\ref{theorem::suboptimal-optimal} does not require typical RIP conditions. Moreover, when $\frac{\lambda}{b} - {\kappa} |{\hat \bz}_j - b|$ for nonzero ${\hat \bz}_j$ and $\frac{\lambda}{b} - {\kappa} b$ are no greater than $0$, or they are small positive numbers, the sub-optimal solution $ {\hat \bz}$ is equal to or very close to the globally optimal solution.
\end{MyRemark}

\section{Accelerated Proximal Gradient Descent by Randomized Algorithms}
In this section, we propose and analyze two randomized algorithms that accelerate PGD for the $\ell^{0}$ sparse approximation problem by random projection. Our first algorithm employs randomized low rank matrix approximation for the dictionary $\bD$, and the second algorithm uses random projection to generate dimensionality-reduced signal and dictionary and then perform PGD on the low-dimensional signal and dictionary. Our theoretical analysis establishes the suboptimality of the solutions obtained by the proposed randomized algorithms.
\subsection{Algorithm}
While Lemma~\ref{lemma::shrinkage-sufficient-decrease} shows that the $\ell^{0}$ sparse approximation (\ref{eq:l0-sparse-appro}) is equivalent to its dictionary-reduced version (\ref{eq:l0-sparse-appro-reduced}) which leads to improved efficiency, we are still facing the computational challenge incurred by dictionary with large dimension $d$ and size $n$ (or large $|\bS|$). The literature has extensively employed randomized algorithms for accelerating the numeral computation of different kinds of matrix optimization problems including low rank approximation and matrix decomposition \citep{Frieze2004-fast-monto-carlo-lowrank,Drineas2004-large-graph-svd,Sarlos2006-large-matrix-random-projection,
Drineas2006-fast-monto-carlo-lowrank,Drineas2008-matrix-decomposition,Mahoney2009-matrix-decomposition,
Drineas2011-fast-least-square-approximation,Lu2013-fast-ridge-regression-random-subsample}. In order to accelerate the numerical computation involved in PGD, we propose two randomized algorithms. The first algorithm adopts the randomized low rank approximation by random projection \citep{Halko2011-random-matrix-decomposition} to obtain a low rank approximation of the dictionary so as to accelerate the computation of gradient for PGD. The second algorithm generates dimensionality-reduced signal and dictionary by random projection and then apply PGD upon the low-dimensional signal and dictionary for improved efficiency. The optimization algorithm for the $\ell^{0}$ sparse approximation problem (\ref{eq:l0-sparse-appro}) by PGD with low rank approximation of $\bD$ via Randomized Matrix Approximation, termed PGD-RMA in this paper, is described in Section~\ref{sec::PGD-RMA}. Bearing the idea of using randomized algorithm for dimension reduction, the algorithm that accelerates PGD via randomized dimension reduction, termed PGD-RDR, is introduced in Section~\ref{sec::PGD-RDR}.  

\subsubsection{Accelerated Proximal Gradient Descent via Randomized Matrix Approximation: PGD-RMA}
\label{sec::PGD-RMA}
The procedure of PGD-RMA is described as follows. A random matrix $\Omega \in \R^{n \times k}$ is computed such that each element $\Omega_{ij}$ is sampled independently from the Gaussian distribution $\cN(0,1)$. With the QR decomposition of $\bD \Omega$, i.e. $\bD \Omega = \bQ \bR$ where $\bQ \in \R^{d \times k}$ is an orthogonal matrix of rank $k$ and $\bR \in \R^{k \times k}$ is an upper triangle matrix. The columns of $\bQ$ form the orthogonal basis for $\bD \Omega$. Then $\bD$ is approximated by projecting $\bD$ onto the range of $\bD \Omega$: $\bQ\bQ^{\top}\bD = \bQ\bW = \tilde \bD$ where $\bW = \bQ^{\top}\bD \in \R^{k \times n}$. Replacing $\bD$ with its low rank approximation $\tilde \bD$, we resort to solve the following reduced $\ell^0$ sparse approximation problem (\ref{eq:l0-sparse-appro-rma})
\begin{small}\begin{align}\label{eq:l0-sparse-appro-rma}
&\mathop {\min }\limits_{{\bz \in \R^n} } \tilde L(\bz) = {\|\bx - \tilde \bD \bz\|_2^2 + {\lambda}\|{\bz}\|_0}
\end{align}\end{small}%
And the first step of PGD (\ref{eq:l0-sa-proximal-step1}) for the original $\ell^0$ sparse approximation problem is reduced to
\begin{small}\begin{align}\label{eq:l0-sa-proximal-step1-reduced}
&\tilde {\bz}^{(t)} = {\bz}^{(t-1)} - \frac{2}{{\tau}s} ({{\tilde \bD}^\top}{\tilde \bD}{{\bz}^{(t-1)}}-{{\tilde \bD}^\top}{\bx}) \\ \nonumber
& ={\bz}^{(t-1)} - \frac{2}{{\tau}s} ({\bW^{\top}} {\bQ^{\top}} \bQ \bW {{\bz}^{(t-1)}}-{\bW^{\top}}{\bQ^{\top}}{\bx})
\end{align}\end{small}%
The complexity of this step is reduced from $\cO(dn)$ to $\cO(dk + nk)$ wherein $k \ll \min\{d,n\}$ and significant efficiency improvement is achieved. Note that the computational cost of QR decomposition for $\bD \Omega$ is less than $2dk^2$, which is acceptable with a small $k$.

The randomized algorithm PGD-RMA is described in Algorithm~\ref{alg:pgd-rma}. The time complexity of PGD-RMA is $\cO(M(dk + nk))$ where $M$ is the number of iterations (or maximum number of iterations), compared to the complexity $\cO(Mdn)$ for the original PGD.
\begin{algorithm}[!ht]
\renewcommand{\algorithmicrequire}{\textbf{Input:}}
\renewcommand\algorithmicensure {\textbf{Output:} }
\caption{Proximal Gradient Descent via Randomized Matrix Approximation (PGD-RMA) for the $\ell^{0}$ Sparse Approximation (\ref{eq:l0-sparse-appro})  }
\label{alg:pgd-rma}
\begin{algorithmic}[1]
\REQUIRE ~~\\
The given signal $\bx \in R^d$, the dictionary $\bD$, the parameter $\lambda$ for the weight of the $\ell^{0}$-norm, maximum iteration number $M$, stopping threshold $\varepsilon$, the initialization ${\bz}^{(0)} \in \R^n$.
\STATE Sample a random matrix $\Omega \in \R^{n \times k}$ by $\Omega_{ij} \sim \cN(0,1)$.
\STATE Compute the QR decomposition of $\bD \Omega$: $\bD \Omega = \bQ \bR$
\STATE Approximate $\bD$ by $\tilde \bD = \bQ\bW$ where $\bW = \bQ^{\top}\bD$
\STATE{Perform PGD with (\ref{eq:l0-sa-proximal-step1-reduced}) and (\ref{eq:l0-sa-proximal-step2})} starting from $t = 1$. The iteration terminates either $\{{\bz}^{(t)}\}_t$ or $\{L({\bz}^{(t)})\}_t$ converges under certain threshold or maximum iteration number is achieved.
\ENSURE Obtain the sparse code $\tilde \bz$ upon the termination of the iterations.
\end{algorithmic}
\end{algorithm}

\subsubsection{Accelerated Proximal Gradient Descent via Random Dimension Reduction: PGD-RDR}
\label{sec::PGD-RDR}
We introduce PGD-RDR which employs random projection to generate low-dimensional signal and dictionary, upon which PGD is applied for improved efficiency. The literature \citep{Frankl:1987-JL-Lemma,Indyk1998-ANN,Zhang2016-sparse-random-convex-concave} extensively considers the random projection that satisfies the following $\ell^2$-norm preserving property, which is closed related to the proof of the Johnson–Lindenstrauss lemma \citep{Dasgupta2003-JL-proof}.
\begin{MyDefinition}\label{def:L2-norm-preseving}
The linear operator $\bT \colon \R^d \to \R^m$ satisfies the $\ell^2$-norm preserving property if there exists constant $c > 0$ such that
\begin{small}\begin{align}
&\Pr\big[(1-\varepsilon)\|\bv\|_2 \le \|\bT \bv\|_2 \le  (1+\varepsilon)\|\bv\|_2\big] \ge 1-2e^{\frac{m\varepsilon^2}{c}}
\end{align}\end{small}%
holds for any fixed $\bv \in \R^d$ and $0 < \varepsilon \le \frac{1}{2}$.
\end{MyDefinition}
The linear operator $\bT$ satisfying the $\ell^2$-norm preserving property can be generated randomly according to uncomplicated distributions. With $\bT^{'} = \sqrt{m} \bT$, it is proved in \citep{Arriaga2006,Achlioptas2003-database-friendly-random-projection} that $\bT$ satisfies the $\ell^2$-norm preserving property, if all the elements of $\bT^{'}$ are sampled independently from the Gaussian distribution $\cN(0,1)$, or uniform distribution over ${\pm 1}$, or the database-friendly distribution described by
\begin{small}\begin{align*}
\bT_{ij}^{'} =
\left\{
\begin{array}
    {r@{\quad:\quad}l}
    \sqrt{3} & {{\rm with probability} \,\, \frac{1}{6}} \\
    \sqrt{0} & {{\rm with probability} \,\, \frac{2}{3}} \\
    -\sqrt{3} & {{\rm with probability} \,\, \frac{1}{6}}
\end{array}
\right., 1 \le i \le m, 1 \le j \le d
\end{align*}\end{small}%
With $m < d$, PGD-RDR first generate the dimensionality-reduced signal and dictionary by $\bar \bx =  \bT \bx$ and $\bar \bD =  \bT \bD$, then
solve the following dimensionality-reduced $\ell^0$ sparse approximation problem
\begin{small}\begin{align}\label{eq:l0-sparse-appro-rdr}
&\mathop {\min }\limits_{{\bz \in \R^n} } \bar L(\bz) = {\|\bar \bx - \bar \bD \bz\|_2^2 + {\lambda}\|{\bz}\|_0}
\end{align}\end{small}%
by PGD. The procedure of PGD-RDR for the $\ell^{0}$ sparse approximation problem (\ref{eq:l0-sparse-appro}) is described in Algorithm~\ref{alg:pgd-rdr}. The time complexity of sampling the random matrix $\bT$ is $\cO(md)$, and the time complexity of the first step of PGD (\ref{eq:l0-sa-proximal-step1}) for gradient descent is reduced from $\cO(dn)$ to $\cO(mn)$. The time complexity of PGD-RDR is $\cO(Mmn)$ where $M$ is the number of iterations (or maximum number of iterations), compared to the complexity $\cO(Mdn)$ for the original PGD. Improvement on the efficiency is achieved with $m < d$.
\begin{algorithm}[!ht]
\renewcommand{\algorithmicrequire}{\textbf{Input:}}
\renewcommand\algorithmicensure {\textbf{Output:} }
\caption{Proximal Gradient Descent via Randomized Dimension Reduction (PGD-RDR) for the $\ell^{0}$ Sparse Approximation (\ref{eq:l0-sparse-appro})  }
\label{alg:pgd-rdr}
\begin{algorithmic}[1]
\REQUIRE ~~\\
The given signal $\bx \in R^d$, the dictionary $\bD$, the parameter $\lambda$ for the weight of the $\ell^{0}$-norm, maximum iteration number $M$, stopping threshold $\varepsilon$, the initialization ${\bz}^{(0)} \in \R^n$.
\STATE Sample a random matrix $\bT \in \R^{m \times n}$ which satisfies the $\ell^2$-norm preserving property in Definition~\ref{def:L2-norm-preseving}, e.g. $\sqrt{m} \bT_{ij} \sim \cN(0,1)$.
\STATE Compute the low-dimensional signal $\bar \bx =  \bT \bx$ and the dimensionality-reduced dictionary $\bar \bD =  \bT \bD$.

\STATE{Perform PGD with (\ref{eq:l0-sa-proximal-step1-reduced}) and (\ref{eq:l0-sa-proximal-step2})} starting from $t = 1$, with $\bx$ and $\bD$ replaced by $\bar \bx$ and $\bar \bD$. The iteration terminates either $\{{\bz}^{(t)}\}_t$ or $\{L({\bz}^{(t)})\}_t$ converges under certain threshold or maximum iteration number is achieved.
\ENSURE Obtain the sparse code $\bar \bz$ upon the termination of the iterations.
\end{algorithmic}
\end{algorithm}

\subsection{Theoretical Analysis}
We analyze the theoretical properties of the proposed PGD-RMA and PGD-RDR in the previous section. For both randomized algorithms, we present the bounded gap between the sub-optimal solution to the reduced $\ell^0$ sparse approximation problem (\ref{eq:l0-sparse-appro-rma}) or (\ref{eq:l0-sparse-appro-rdr}) and the globally optimal solution $\bz^*$ to the original problem.

\subsection{Analysis for PGD-RMA}
\citep{Halko2011-random-matrix-decomposition} proved that the approximation $\tilde \bD$ is close to $\bD$ in terms of the spectral norm:
\begin{MyLemma}\label{lemma::D-approx}
(\textit{Corollary $10.9$ by \citep{Halko2011-random-matrix-decomposition} })
Let $k_0 \ge 2$ and $p = k-k_0 \ge 4$, then probability at least $1-6e^{-p}$, then the spectral norm of $\bD - \hat \bD$ is bounded by
\begin{align}\label{eq:D-appro}
&\|\bD - \hat \bD\|_2 \le C_{k,k_0}
\end{align}
where
\begin{align}\label{eq:C-k-k0}
&C_{k,k_0} = \big(1+17\sqrt{1+\frac{k_0}{p}}\big) \sigma_{k_0+1} + \frac{8\sqrt{k}}{p+1} (\sum\limits_{j > k_0} \sigma_j^2)^{\frac{1}{2}}
\end{align}
$\sigma_1 \ge \sigma_2 \ge \ldots$ are the singular values of $\bD$.
\end{MyLemma}

Let $\tilde \bz$ be the globally optimal solution to (\ref{eq:l0-sparse-appro-rma}), $\tilde \bS = {\rm supp}(\tilde {\bz})$, $\tilde Q(\bz) = \|\bx - \tilde \bD {\bz}\|_2^2$, ${\bz}^{(0)}$ be the initialization for PGD for the optimization of the reduced problem (\ref{eq:l0-sparse-appro-rma}). We have the following theorem showing the upper bound for the gap between $\tilde \bz$ and ${\bz}^*$.
\begin{MyTheorem}\label{theorem::optimal-rma}
(\textit{Optimal solution to the reduced problem (\ref{eq:l0-sparse-appro-rma}) is close to the that to the original problem})
Let $\bG = \tilde \bS \cup \bS^*$. Suppose $\bD_{\bG}$ is not singular with $\tau_0 \triangleq \sigma_{\min}(\bD_{\bG}) > 0$, $2\tau_0^2 > \tau > 0$. Then with probability at least $1-6e^{-p}$,
\begin{small}\begin{align}\label{eq:optimal-rp}
&\|{\bz}^*-{\tilde \bz}\|_2 \nonumber \\
&\le \frac{1}{2\tau_0^2-\tau}\bigg(\big(\sum\limits_{j \in {\bG} \cap \tilde \bS} (\max\{0,\frac{\lambda}{b} - {\kappa} |{\tilde \bz}_j - b| \})^2 + \nonumber \\
&\sum\limits_{j \in {\bG} \setminus \tilde \bS} (\max\{0, \frac{\lambda}{b} - {\kappa} b\})^2 \big)^{\frac{1}{2}} + 2C_{k,k_0}M_0 (2\sigma_{\max}(\bD) + C_{k,k_0}) + 2C_{k,k_0}\|\bx\|_2 \bigg)
\end{align}\end{small}%
where $M_0 = \frac{\|\bx\|_2 + \sqrt{{\tilde L}({\bz}^{(0)})}} {\tau_0}$, and $b$ satisfies
\begin{small}\begin{align}\label{eq:b-cond-rp}
&0< b < \min\{\min_{j \in {\tilde \bS}} | {\tilde \bz}_j|, \max_{k \notin {\tilde \bS}} \frac{\lambda }{ (\frac{\partial {\tilde Q}}{\partial {\bz_k}}|_{\bz =  {\tilde \bz}}-\lambda)_{+}},
\min_{j \in {\bS^*}} | {\bz_j}^*|, \max_{k \notin {\bS^*}} \frac{\lambda}{(\frac{\partial Q}{\partial {\bz_k}}|_{\bz = {\bz}^*}-\lambda)_{+}}  \}
\end{align}\end{small}
\end{MyTheorem}

Let $\bS = {\rm supp}({\bz}^{(0)})$. According to Theorem~\ref{theorem::suboptimal-optimal} and Theorem~\ref{theorem::optimal-rma}, we have the bounded gap between the sub-optimal solution to the reduced problem (\ref{eq:l0-sparse-appro-rma}) and the globally optimal solution ${\bz}^*$ to the original problem (\ref{eq:l0-sparse-appro}).
\begin{MyTheorem}\label{theorem::suboptimal-optimal-rp}
(\textit{Sub-optimal solution to the reduced problem (\ref{eq:l0-sparse-appro-rma}) is close to the globally optimal solution to the original problem})
Choose ${\bz}^{(0)}$ such that $\|\bx - {\tilde \bD} {\bz}^{(0)}\|_2 \le 1$, and $s > \max\{2|\bS|, \frac{2(1+{\lambda}|\bS|)}{\lambda \tau}\}$, if ${\tilde \bD}_{\bS}$ is nonsingular, then the sequence $\{{\bz}^{(t)}\}_t$ generated by PDG for the reduced problem (\ref{eq:l0-sparse-appro-rma}) converges to a critical point of ${\tilde L}(\bz)$, denoted by $\hat {\tilde \bz}$. Let $\hat \bS = {\rm supp}(\hat {\tilde \bz})$, $\bF = ({\hat \bS} \setminus \tilde \bS) \cup (\tilde \bS \setminus {\hat \bS})$, $\bG = \tilde \bS \cup \bS^*$. Suppose ${\tilde \bD}_{\hat \bS \cup \tilde \bS}$ is not singular with $\kappa_0 \triangleq \sigma_{\min}({\tilde \bD}_{\hat \bS \cup \tilde \bS}) > 0$, and $\kappa_0^2 > \kappa > 0$; $\bD_{\bG}$ is not singular with $\tau_0 \triangleq \sigma_{\min}(\bD_{\bG}) > 0$, $2\tau_0^2 > \tau > 0$. Then with probability at least $1-6e^{-p}$,
\begin{small}\begin{align}\label{eq:optimal-rp}
&\|{\bz}^*-{\hat {\tilde \bz}}\|_2 \le b_1 + b_2,
\end{align}\end{small}%
where
\begin{small}\begin{align*}
&b_1 = \frac{\big(\sum\limits_{j \in \bF \cap \hat \bS} (\max\{0,\frac{\lambda}{b} - {\kappa} |{\hat {\tilde \bz}}_j - b|\})^2 +
\sum\limits_{j \in \bF \setminus \hat \bS} (\max\{0,\frac{\lambda}{b} - {\kappa} b\})^2 \big)^{\frac{1}{2}}}{2\kappa_0^2-\kappa} \\
&b_2 = \frac{1}{2\tau_0^2-\tau}\bigg(\big(\sum\limits_{j \in {\bG} \cap \tilde \bS} (\max\{0,\frac{\lambda}{b} - {\kappa} |{\tilde \bz}_j - b| \})^2 +\sum\limits_{j \in {\bG} \setminus \tilde \bS} (\max\{0, \frac{\lambda}{b} - {\kappa} b\})^2 \big)^{\frac{1}{2}} \nonumber \\
&+ 2C_{k,k_0}M_0 (2\sigma_{\max}(\bD) + C_{k,k_0}) + 2C_{k,k_0}\|\bx\|_2 \bigg)
\end{align*}
\end{small}%
$M_0 = \frac{\|\bx\|_2 + \sqrt{{\tilde L}({\bz}^{(0)})}} {\tau_0}$, and $b$ satisfies
\begin{small}\begin{align}\label{eq:b-cond-rp}
&0< b < \min\{\min_{j \in {\tilde \bS}} | {\tilde \bz}_j|, \max_{k \notin {\tilde \bS}} \frac{\lambda }{ (\frac{\partial {\tilde Q}}{\partial {\bz_k}}|_{\bz =  {\tilde \bz}}-\lambda)_{+}},
\min_{j \in {\bS^*}} | {\bz_j}^*|, \max_{k \notin {\bS^*}} \frac{\lambda}{(\frac{\partial Q}{\partial {\bz_k}}|_{\bz = {\bz}^*}-\lambda)_{+}}, \min_{j \in {\hat \bS}} | {\hat {\tilde \bz}}_j|,
\max_{k \notin {\hat \bS}} \frac{\lambda }{ (\frac{\partial {\tilde Q}}{\partial {\bz_k}}|_{\bz =  {\hat {\tilde \bz}}}-\lambda)_{+}} \}
\end{align}\end{small}%
\end{MyTheorem}

\subsection{Analysis for PGD-RDR}

Let $\bar \bz$ be the globally optimal solution to (\ref{eq:l0-sparse-appro-rdr}), $\bar \bS = {\rm supp}(\bar {\bz})$, $\bar Q(\bz) = \|\bar \bx - \bar \bD {\bz}\|_2^2$, ${\bz}^{(0)}$ be the initialization for PGD for the optimization of the reduced problem (\ref{eq:l0-sparse-appro-rdr}). We have the following theorem showing the upper bound for the gap between $\bar \bz$ and ${\bz}^*$.
\begin{MyTheorem}\label{theorem::optimal-rdr}
(\textit{Optimal solution to the dimensionality-reduced problem (\ref{eq:l0-sparse-appro-rdr}) is close to the that to the original problem})
Let $\bH = \bar \bS \cup \bS^*$. Suppose $\bD_{\bH}$ is not singular with $\eta_0 \triangleq \sigma_{\min}(\bD_{\bG}) > 0$, $2\eta_0^2 > \eta > 0$. If $\bT$ satisfies the $\ell^2$-norm preserving property in Definition~\ref{def:L2-norm-preseving}, $m \ge 4c\log{\frac{4}{\delta}}$, then with probability at least $1 - \delta$,
\begin{small}\begin{align}\label{eq:optimal-rp}
&\|{\bz}^*-{\bar \bz}\|_2 \nonumber \\
&\le \frac{1}{2\eta_0^2-\eta}\bigg(\big(\sum\limits_{j \in {\bH} \cap \bar \bS} (\max\{0,\frac{\lambda}{b} - {\kappa} |{\bar \bz}_j - b| \})^2 + \nonumber \\
&\sum\limits_{j \in {\bH} \setminus \bar \bS} (\max\{0, \frac{\lambda}{b} - {\kappa} b\})^2 \big)^{\frac{1}{2}} + 2\|\bD\|_F{M_1}\sqrt{\frac{c}{m}\log{\frac{4}{\delta}}} (\sigma_{\max}(\bD)+1) \bigg)
\end{align}\end{small}%
where $M_1 = \frac{\|\bx\|_2 + \sqrt{{\bar L}({\bz}^{(0)})}} {\eta_0}$, and $b$ satisfies
\begin{small}\begin{align}\label{eq:b-cond-rp}
&0< b < \min\{\min_{j \in {\bar \bS}} | {\bar \bz}_j|, \max_{k \notin {\bar \bS}} \frac{\lambda }{ (\frac{\partial {\bar Q}}{\partial {\bz_k}}|_{\bz =  {\bar \bz}}-\lambda)_{+}},
\min_{j \in {\bS^*}} | {\bz_j}^*|, \max_{k \notin {\bS^*}} \frac{\lambda}{(\frac{\partial Q}{\partial {\bz_k}}|_{\bz = {\bz}^*}-\lambda)_{+}}  \}
\end{align}\end{small}
\end{MyTheorem}

Similar to the analysis for PGD-RMA, we let $\bS = {\rm supp}({\bz}^{(0)})$. Combining Theorem~\ref{theorem::suboptimal-optimal} and Theorem~\ref{theorem::optimal-rdr}, we have the bounded gap between the sub-optimal solution to the dimensionality-reduced problem (\ref{eq:l0-sparse-appro-rdr}) and the globally optimal solution ${\bz}^*$ to the original problem (\ref{eq:l0-sparse-appro}).
\begin{MyTheorem}\label{theorem::suboptimal-optimal-rp}
(\textit{Sub-optimal solution to the dimensionality-reduced problem (\ref{eq:l0-sparse-appro-rdr}) is close to the globally optimal solution to the original problem})
Choose ${\bz}^{(0)}$ such that $\|\bx - {\bar \bD} {\bz}^{(0)}\|_2 \le 1$, and $s > \max\{2|\bS|, \frac{2(1+{\lambda}|\bS|)}{\lambda \eta}\}$, if ${\bar \bD}_{\bS}$ is nonsingular, then the sequence $\{{\bz}^{(t)}\}_t$ generated by PDG for the dimensionality-reduced problem (\ref{eq:l0-sparse-appro-rdr}) converges to a critical point of ${\bar L}(\bz)$, denoted by $\hat {\bar \bz}$. Let $\hat \bS = {\rm supp}(\hat {\bar \bz})$, $\bF = ({\hat \bS} \setminus \bar \bS) \cup (\bar \bS \setminus {\hat \bS})$, $\bH = \bar \bS \cup \bS^*$. Suppose ${\bar \bD}_{\hat \bS \cup \bar \bS}$ is not singular with $\kappa_0 \triangleq \sigma_{\min}({\bar \bD}_{\hat \bS \cup \bar \bS}) > 0$, and $\kappa_0^2 > \kappa > 0$; $\bD_{\bH}$ is not singular with $\eta_0 \triangleq \sigma_{\min}(\bD_{\bH}) > 0$, $2\eta_0^2 > \eta > 0$. If $\bT$ satisfies the $\ell^2$-norm preserving property in Definition~\ref{def:L2-norm-preseving}, $m \ge 4c\log{\frac{4}{\delta}}$, then with probability at least $1 - \delta$,
\begin{small}\begin{align}\label{eq:optimal-rp}
&\|{\bz}^*-{\hat {\bar \bz}}\|_2 \le b_1 + b_2,
\end{align}\end{small}%
where
\begin{small}\begin{align*}
&b_1 = \frac{\big(\sum\limits_{j \in \bF \cap \hat \bS} (\max\{0,\frac{\lambda}{b} - {\kappa} |{\hat {\bar \bz}}_j - b|\})^2 +
\sum\limits_{j \in \bF \setminus \hat \bS} (\max\{0,\frac{\lambda}{b} - {\kappa} b\})^2 \big)^{\frac{1}{2}}}{2\kappa_0^2-\kappa} \\
&b_2 = \frac{1}{2\eta_0^2-\eta}\bigg(\big(\sum\limits_{j \in {\bH} \cap \bar \bS} (\max\{0,\frac{\lambda}{b} - {\kappa} |{\bar \bz}_j - b| \})^2 + \nonumber \\
&\sum\limits_{j \in {\bH} \setminus \bar \bS} (\max\{0, \frac{\lambda}{b} - {\kappa} b\})^2 \big)^{\frac{1}{2}} + 2\|\bD\|_F{M_1}\sqrt{\frac{c}{m}\log{\frac{4}{\delta}}} (\sigma_{\max}(\bD)+1) \bigg)
\end{align*}
\end{small}%
$M_0 = \frac{\|\bx\|_2 + \sqrt{{\bar L}({\bz}^{(0)})}} {\eta_0}$, and $b$ satisfies
\begin{small}\begin{align}\label{eq:b-cond-rp}
&0< b < \min\{\min_{j \in {\bar \bS}} | {\bar \bz}_j|, \max_{k \notin {\bar \bS}} \frac{\lambda }{ (\frac{\partial {\bar Q}}{\partial {\bz_k}}|_{\bz =  {\bar \bz}}-\lambda)_{+}},
\min_{j \in {\bS^*}} | {\bz_j}^*|, \max_{k \notin {\bS^*}} \frac{\lambda}{(\frac{\partial Q}{\partial {\bz_k}}|_{\bz = {\bz}^*}-\lambda)_{+}}, \min_{j \in {\hat \bS}} | {\hat {\bar \bz}}_j|,
\max_{k \notin {\hat \bS}} \frac{\lambda }{ (\frac{\partial {\bar Q}}{\partial {\bz_k}}|_{\bz =  {\hat {\bar \bz}}}-\lambda)_{+}} \}
\end{align}\end{small}%
\end{MyTheorem}

The detailed proofs of the theorems and lemmas are included in Section~\ref{sec::supplementary}. Note that we slightly abuse the notation of $\bF$, $\hat \bS$, $\kappa$ and $\kappa_0$ in the analysis for PGD-RMA and PGD-RDR with no confusion. To the best of our knowledge, our theoretical results are among the very few results for the provable randomized efficient algorithms for the nonsmooth and nonconvex $\ell^0$ sparse approximation problem.
\section{Conclusions}
We propose to use proximal gradient descent (PGD) to obtain a sub-optimal solution to the $\ell^{0}$ sparse approximation problem. Our theoretical analysis renders the bound for the $\ell^{2}$-distance between the sub-optimal solution and the globally optimal solution, under conditions weaker than Restricted Isometry Property (RIP). To the best of our knowledge, this is the first time that such gap between sub-optimal solution and globally optimal solution is obtained under our mild conditions. Moreover, we propose provable randomized algorithms, namely proximal gradient descent via Randomized Matrix Approximation (PGD-RMA) and proximal gradient descent via Random Dimension Reduction (PGD-RDR), to accelerate the ordinary optimization by PGD.
\bibliography{mybib}

\section{Proofs}
\label{sec::supplementary}
\subsection{Proof of Lemma~\ref{lemma::shrinkage-sufficient-decrease}}
\begin{proof}
We prove this Lemma by mathematical induction.

When $t=1$, we first show that ${\rm supp} ({\bz}^{(1)}) \subseteq {\rm supp} ({\bz}^{(0)})$, i.e. the support of $\bz$ shrinks after the first iteration. To see this,
$\tilde {\bz}^{(t)} = {\bz}^{(t-1)} - \frac{2}{{\tau}s} ({\bD^\top}{\bD}{{\bz}^{(t-1)}}-{\bD^\top}{\bx})$.

Since $\|\bx - \bD {\bz}^{(t-1)}\|_2^2 \le 1$, let $\bg^{(t-1)} = - \frac{2}{{\tau}s} ({\bD^\top}{\bD}{{\bz}^{(t-1)}}-{\bD^\top}{\bx})$, then
\begin{small}\begin{align*}
&|\tilde {\bz_j}^{(t)}| \le \|\bg^{(t-1)}\|_{\infty} \le \frac{2}{{\tau}s} \|{\bD^\top}({\bD}{{\bz}^{(t-1)}}-{\bx})\|_{\infty} \le \frac{2}{{\tau}s}
\end{align*}\end{small}
where $j$ is the index for any zero element of ${\bz}^{(t-1)}$, $1 \le j \le n, j \notin {\rm supp}({\bz}^{(t-1)})$. Now $|{\tilde {\bz_j}^{(t)}}| < \sqrt{\frac{2\lambda}{{\tau}s}}$, and it follows that ${\bz_j}^{(t)} = 0$ due to the update rule (\ref{eq:l0-sa-proximal-step2}). Therefore, the zero elements of ${\bz}^{(t-1)}$ remain unchanged in ${\bz}^{(t)}$, and ${\rm supp} ({\bz}^{(t)}) \subseteq {\rm supp} ({\bz}^{(t-1)})$ for $t=1$.

Let $Q_{\bS}(\by) = \|\bx - \bD_{\bS} \by\|_2^2$ for $\by \in \R^{|\bS|}$, then we show that $s > 2|\bS|$ is the Lipschitz constant for the gradient of function $Q_{\bS}$. To see this, we have
\begin{small}\begin{align*}
\sigma_{\max}({\bD_{\bS}^\top}{\bD_{\bS}}) = \big(\sigma_{\max}(\bD_{\bS})\big)^2 \le {\rm Tr}({\bD_{\bS}^\top}{\bD_{\bS}}) = |\bS|
\end{align*}\end{small}
Also, $\nabla Q_{\bS}(\by) = 2 ({\bD_{\bS}^\top}{\bD_{\bS}}{\by}-{\bD_{\bS}^\top}{\bx})$, and
\begin{small}\begin{align}\label{eq:lemma1-proof-seg1}
& \|\nabla Q_{\bS}(\by) - \nabla Q_{\bS}(\bz)\|_2 = 2 \| {\bD_{\bS}^\top}{\bD_{\bS}}({\by}-{\bz})\|_2 \\
&\le 2 \sigma_{\max}({\bD_{\bS}^\top}{\bD_{\bS}}) \cdot \|({\by}-{\bz})\|_2  \nonumber \\
& \le 2|\bS| \|({\by}-{\bz})\|_2 < s \|({\by}-{\bz})\|_2 \nonumber
\end{align}
\end{small}
Note that when $t=1$, since $${\bz}^{(t)} = \argmin \limits_{\bv \in \R^n} {\frac{{\tau}s}{2}\|\bv - {\tilde {\bz}^{(t)}}\|_2^2 + {\lambda}\|\bv\|_0}$$ we have
\begin{small}\begin{align}\label{eq:lemma1-proof-seg2}
& \frac{{\tau}s}{2}\|{\bz}^{(t)} - {\tilde {\bz}^{(t)}}\|_2^2 + {\lambda}\|{\bz}^{(t)}\|_0 \\
&\le \frac{{\tau}s}{2}\|\frac{\nabla Q({\bz}^{(t-1)})}{{\tau}s}\|_2^2 + {\lambda}\|{\bz}^{(t-1)}\|_0 \nonumber
\end{align}\end{small}
which is equivalent to
\begin{small}\begin{align}\label{eq:lemma1-proof-seg3}
& \langle \nabla Q_{\bS} ({\bz_{\bS}}^{(t-1)}), {\bz_{\bS}}^{(t)} - {\bz_{\bS}}^{(t-1)}\rangle + \frac{{\tau}s}{2} \|{\bz}^{(t)} - {\bz}^{(t-1)}\|_2^2 \\
& + {\lambda}\|{\bz}^{(t)}\|_0 \le {\lambda}\|{\bz}^{(t-1)}\|_0 \nonumber
\end{align}\end{small}
due to the fact that
\begin{small}\begin{align*}
&\langle \nabla Q ({\bz}^{(t-1)}), {\bz}^{(t)} - {\bz}^{(t-1)}\rangle =
\langle \nabla Q_{\bS} ({\bz_{\bS}}^{(t-1)}), {\bz_{\bS}}^{(t)} - {\bz_{\bS}}^{(t-1)}\rangle
\end{align*}\end{small}
Also, since $s$ is the Lipschitz constant for $\nabla Q_{\bS}$,
\begin{small}\begin{align}\label{eq:lemma1-proof-seg4}
& Q_{\bS}({\bz_{\bS}}^{(t)}) \le  Q_{\bS}({\bz_{\bS}}^{(t-1)}) + \langle \nabla Q_{\bS} ({\bz_{\bS}}^{(t-1)}), {\bz_{\bS}}^{(t)} - {\bz_{\bS}}^{(t-1)}\rangle \\
& + \frac{s}{2} \|{\bz_{\bS}}^{(t)} - {\bz_{\bS}}^{(t-1)}\|_2^2 \nonumber
\end{align}
\end{small}
Combining (\ref{eq:lemma1-proof-seg3}) and (\ref{eq:lemma1-proof-seg4}) and note that $\|{\bz_{\bS}}^{(t)} - {\bz_{\bS}}^{(t-1)}\|_2 = \|{\bz}^{(t)} - {\bz}^{(t-1)}\|_2$, $Q_{\bS}({\bz_{\bS}}^{(t)}) = Q({\bz}^{(t)})$ and $Q_{\bS}({\bz_{\bS}}^{(t-1)}) = Q({\bz}^{(t-1)})$, we have
\begin{small}\begin{align}\label{eq:lemma1-proof-seg5}
&Q({\bz}^{(t)}) + {\lambda}\|{\bz}^{(t)}\|_0 \le Q({\bz}^{(t-1)}) + {\lambda}\|{\bz}^{(t-1)}\|_0 \\
& - \frac{(\tau-1)s}{2} \|{\bz}^{(t)} - {\bz}^{(t-1)}\|_2^2 \nonumber
\end{align}
\end{small}
Now (\ref{eq:support-shrinkage}) and (\ref{eq:l0graph-proximal-sufficient-decrease}) are verified for $t = 1$. Suppose (\ref{eq:support-shrinkage}) and (\ref{eq:l0graph-proximal-sufficient-decrease}) hold for all $t \ge t_0$ with $t_0 \ge 1$.
Since $\{L({\bz}^{(t)})\}_{t=1}^{t_0}$ is decreasing, we have
\begin{small}\begin{align*}
&L({\bz}^{(t_0)}) = \|\bx - \bD {\bz}^{(t_0)}\|_2^2 + {\lambda}\|{\bz}^{(t_0)}\|_0 \\
&\le \|\bx - \bD {\bz}^{(0)}\|_2^2 + {\lambda}\|{\bz}^{(0)}\|_0 \le 1 + {\lambda}|\bS|  \nonumber
\end{align*}\end{small}
which indicates that $\|\bx - \bD {\bz}^{(t_0)}\|_2 \le \sqrt{1+{\lambda}|\bS|}$.
When $t = t_0+1$,
\begin{small}\begin{align*}
&|\tilde {\bz_j}^{(t)}| \le \|\bg^{(t-1)}\|_{\infty} \le \frac{2}{{\tau}s} \|{\bD^\top}({\bD}{{\bz}^{(t-1)}}-{\bx})\|_{\infty} \\
&\le \frac{2}{{\tau}s} \sqrt{1+{\lambda}|\bS|}
\end{align*}\end{small}
where $j$ is the index for any zero element of ${\bz}^{(t-1)}$, $1 \le j \le n, j \notin {\rm supp}({\bz}^{(t-1)})$. Now $|{\tilde {\bz_j}^{(t)}}| < \sqrt{\frac{2\lambda}{{\tau}s}}$, and it follows that and ${\bz_j}^{(t)} = 0$ due to the update rule in (\ref{eq:l0-sa-proximal-step2}). Therefore, the zero elements of ${\bz}^{(t-1)}$ remain unchanged in ${\bz_j}^{(t)}$, and ${\rm supp} ({\bz}^{(t)}) \subseteq {\rm supp} ({\bz}^{(t-1)}) \subseteq \bS$ for $t=t_0+1$. Moreover, similar to the case when $t = 1$, we can derive (\ref{eq:lemma1-proof-seg3}), (\ref{eq:lemma1-proof-seg4}) and (\ref{eq:lemma1-proof-seg5}), so that the support shrinkage (\ref{eq:support-shrinkage}) and decline of the objective (\ref{eq:l0graph-proximal-sufficient-decrease}) are verified for $t=t_0+1$. It follows that the claim of this lemma holds for all $t \ge 1$.

Since the sequence $\{L({\bz}^{(t)})\}_t$ is deceasing with lower bound $0$, it must converge.
\end{proof}

\subsection{Proof of Lemma~\ref{lemma::PGD-convergence}}
\begin{proof}
We first prove that the sequences $\{{\bz}^{(t)}\}_t$ is bounded for any $1 \le i \le n$. In the proof of Lemma~\ref{lemma::shrinkage-sufficient-decrease}, it is proved that
\begin{small}\begin{align*}
&L({\bz}^{(t)}) = \|\bx - \bD {\bz}^{(t)}\|_2^2 + {\lambda}\|{\bz}^{(t)}\|_0 \\
&\le \|\bx - \bD {\bz}^{(0)}\|_2^2 + {\lambda}\|{\bz}^{(0)}\|_0 \le 1 + {\lambda}|\bS|  \nonumber
\end{align*}\end{small}
for $t \ge 1$. Therefore, $\|\bx - \bD {\bz}^{(t)}\|_2 \le \sqrt{1 + {\lambda}|\bS|}$ and it follows that $\|\bD {\bz}^{(t)}\|_2^2 \le (1 + \sqrt{1+{\lambda}|\bS|})^2$. Since ${\rm supp}({\bz}^{(t)}) \subseteq \bS$ for $t \ge 0$ due to Lemma~\ref{lemma::shrinkage-sufficient-decrease},
\begin{small}\begin{align*}
&(1 + \sqrt{1+{\lambda}|\bS|})^2 \ge \|\bD {\bz}^{(t)}\|_2  = \|\bD_{\bS} {\bz_{\bS}}^{(t)}\|_2 \\
&\ge \sigma_{\min}({\bD_{\bS}}^{\top}{\bD_{\bS}})
\| {\bz_{\bS}}^{(t)}\|_2^2 = \sigma_{\min}({\bD_{\bS}}^{\top}{\bD_{\bS}}) \|{\bz}^{(t)}\|_2^2
\end{align*}\end{small}
Since $\bD_{\bS}$ is nonsingular, we have $\sigma_{\min}({\bD_{\bS}}^{\top}{\bD_{\bS}})  = (\sigma_{\min}(\bD))^2$ and it follows that ${\bz}^{(t)}$ is bounded: $\|{\bz}^{(t)}\|_2^2 \le \frac{(1 + \sqrt{1+{\lambda}|\bS|})^2}{(\sigma_{\min}(\bD))^2}$.
In addition, since $\ell^{0}$-norm function $\|\cdot\|_0$ is a semi-algebraic function, therefore, according to Theorem $1$ in \cite{BoltePAL2014}, $\{{\bz}^{(t)}\}_t$ converges to a critical point of $L(\bz)$, denoted by $\hat {\bz}$.
\end{proof}

\subsection{Proof of Lemma~\ref{lemma::equivalence-to-capped-l1}}
\begin{proof}
Let $\hat \bv = 2{\bD^{\top}}({\bD} {\hat {\bz}} - \bx ) + {\dot \bR} (\hat {\bz};b)$, . For for $j \in {\hat \bS}$, since $\hat {\bz}$ is a critical point of $L(\bz) = {\|\bx - \bD \bz\|_2^2 + {\lambda}\|{\bz}\|_0}$. then $\frac{\partial Q}{\partial {\bz_j}} |_{\bz = \hat {\bz}} = 0$ because $\frac{\partial \|\bz\|_0}{\partial {\bz_j}} |_{\bz = \hat {\bz}} = 0$ . Note that $\min_{j \in {\hat \bS}} |\hat {\bz_j}| > b$, so
$\frac{\partial \bR}{\partial \bz_j}|_{\bz = \hat {\bz}} = 0$, and it follows that $\hat \bv_j = 0$.

For $j \notin {\hat \bS}$, since $\frac{d {R}}{d \bz_j} (\hat \bz_j+;b) = \frac{\lambda}{b}$ and $\frac{d {R}}{d \bz_j} (\hat \bz_j-;b) = -\frac{\lambda}{b}$,
$\frac{\lambda}{b} > \max_{j \notin {\hat \bS}} |\frac{\partial Q}{\partial {\bz_j}}|_{\bz = \hat {\bz}}|$, we can choose the $j$-th element of ${\dot \bR} (\hat {\bz};b)$ such that $\hat \bv_j = 0$. Therefore, $\|\hat \bv\|_2 = 0$, and $\hat {\bz}$ is a local solution to the problem (\ref{eq:capped-l1-problem}).

Now we prove that ${\bz}^*$ is also a local solution to (\ref{eq:capped-l1-problem}). Let $\bv^* = 2{\bD^{\top}}({\bD} {\bz}^* - \bx ) + {\dot \bR} ({\bz}^*;b)$, and $Q$ is defined as before. For $j \in \bS^*$, since ${\bz}^*$ is the global optimal solution to problem (\ref{eq:l0-sparse-appro}), we also have $\frac{\partial Q}{\partial {\bz_j}} |_{\bz = {\bz}^*}  = 0$. If it is not the case and $\frac{\partial Q}{\partial {\bz_j}} |_{\bz = {\bz}^*}  \neq 0$, then we can change ${\bz_j}$ by a small amount in the direction of the gradient $\frac{\partial Q}{\partial {\bz_j}}$ at the point ${\bz} = {\bz}^*$ and still make ${\bz_j} \neq 0$, leading to a smaller value of the objective $L(\bz)$.

Note that $\min_{j \in {\bS^*}} | {\bz_j}^*| > b$, so
$\frac{\partial \bR}{\partial \bz_j}|_{\bz = \hat {\bz}} = 0$, and it follows that $\bv_j^* = 0$.

For $j \notin \bS^*$, since $\frac{\lambda}{b} > \max_{j \notin {\hat \bS}} |\frac{\partial Q}{\partial {\bz_j}}|_{\bz =  {\bz}^*}|$, we can choose the $j$-th element of ${\dot \bR} ({\bz}^*;b)$ such that $\bv_j^* = 0$. It follows that $\|\bv^*\|_2 = 0$, and ${\bz}^*$ is also a local solution to the problem (\ref{eq:capped-l1-problem}).
\end{proof}

\subsection{Proof of Theorem~\ref{theorem::suboptimal-optimal}}
\begin{proof}
According to Lemma~\ref{lemma::equivalence-to-capped-l1}, both ${\hat \bz}$ and ${\bz}^*$ are local solutions to problem (\ref{eq:capped-l1-problem}). In the following text, let $\bbeta_{\bI}$ indicates a vector whose elements are those of $\bbeta$ with indices in $\bI$. Let $\Delta = {\bz}^*-{\hat \bz}$, $\tilde \Delta = {\dot \bP}({\bz}^*) - {\dot \bP}(\hat {\bz})$. By Lemma~\ref{lemma::equivalence-to-capped-l1}, we have
\begin{small}\begin{align*}
&\| 2{\bD^{\top}}{\bD}\Delta  + \tilde \Delta\|_2  = 0
\end{align*}\end{small}
It follows that
\begin{small}\begin{align*}
& 2\Delta^{\top}{\bD^{\top}}{\bD}\Delta  + \Delta^{\top} \tilde \Delta  \le \|\Delta\|_2 \| 2{\bD^{\top}}{\bD}\Delta  + \tilde \Delta\|_2 = 0
\end{align*}\end{small}
Also, by the proof of Lemma~\ref{lemma::equivalence-to-capped-l1}, for $k \in {\hat \bS} \cap \bS^*$, since $({\bD^{\top}}{\bD}\Delta)_k = 0$ we have $\tilde \Delta_k = 0$. We now present another property on any nonconvex function $P$ using the degree of nonconvexity in Definition~\ref{def::degree-nonconvexity}: $\theta(t,\kappa):= \sup_s \{ -{\rm sgn}(s-t) ({\dot P}(s;b) - {\dot P}(t;b)) - \kappa |s-t|\}$ on the regularizer $\bP$. For any $s,t \in \R$, we have
\begin{small}\begin{align*}
& -{\rm sgn}(s-t) \big( {\dot P}(s;b) - {\dot P}(t;b) \big) - \kappa |s-t| \le \theta(t,\kappa)
\end{align*}\end{small}
by the definition of $\theta$. It follows that
\begin{small}\begin{align}\label{eq:suboptimal-optimal-seg1}
& \theta(t,\kappa) |s-t| \ge -(s-t)\big( {\dot P}(s;b) - {\dot P}(t;b)  \big) - \kappa (s-t)^2 \nonumber \\
& -(s-t)\big( {\dot P}(s;b) - {\dot P}(t;b)  \big) \le \theta(t,\kappa) |s-t| + \kappa (s-t)^2
\end{align}\end{small}
Applying (\ref{eq:suboptimal-optimal-seg1}) with $P = P_j$ for $j = 1,\ldots,n$, we have
\begin{small}\begin{align}\label{eq:suboptimal-optimal-seg2}
&2\Delta^{\top}{\bD^{\top}}{\bD}\Delta \le -\Delta^{\top} {\tilde \Delta} = -\Delta_{\bF}^{\top} {\tilde \Delta_{\bF }}
-\Delta_{{\hat \bS} \cap \bS^*}^{\top} {\tilde \Delta_{{\hat \bS} \cap \bS^* }} \nonumber \\
& \le |{\bz_{\bF}}^* - {\hat \bz_{\bF}}|^{\top} \theta(\hat \bz_{\bF},\kappa) + \kappa \|{\bz_{\bF}}^* - {\hat \bz_{\bF}}\|_2^2 + \|\Delta_{{\hat \bS} \cap \bS^*}\|_2
{\|{\tilde \Delta_{{\hat \bS} \cap \bS^* }}\|_2}\nonumber \\
& \le \|\theta(\hat \bz_{\bF},\kappa)\|_2 \|{\bz_{\bF}}^* - {\hat \bz_{\bF}}\|_2 + \kappa \|{\bz_{\bF}}* - {\hat \bz_{\bF}}\|_2^2 + \|\Delta\|_2{\|{\tilde \Delta_{{\hat \bS} \cap \bS^* }}\|_2}\nonumber \\
&\le \|\theta({\hat \bz_{\bF}},\kappa)\|_2 \|\Delta\|_2 + \kappa \|\Delta\|_2^2 + \|\Delta\|_2{\|{\tilde \Delta_{{\hat \bS} \cap \bS^* }}\|_2}
\end{align}\end{small}
On the other hand, $\Delta^{\top}{\bD^{\top}}{\bD}\Delta \ge \kappa_0^2 \|\Delta\|_2^2$. It follows from (\ref{eq:suboptimal-optimal-seg2}) that
\begin{small}\begin{align*}
& 2\kappa_0^2 \|\Delta\|_2^2 \le \|\theta(\hat \bz_{\bF},\kappa)\|_2 \|\Delta\|_2 + \kappa \|\Delta\|_2^2 + \|\Delta\|_2 {\|{\tilde \Delta_{{\hat \bS} \cap \bS^* }}\|_2}
\end{align*}\end{small}
When $\|\Delta\|_2 \neq 0$, we have
\begin{small}\begin{align}\label{eq:suboptimal-optimal-seg3}
&2\kappa_0^2 \|\Delta\|_2 \le \|\theta(\hat \bz_{\bF},\kappa)\|_2 + \kappa \|\Delta\|_2 + {\|{\tilde \Delta_{{\hat \bS} \cap \bS^* }}\|_2}\nonumber \\
& \Rightarrow \|\Delta\|_2 \le \frac{\|\theta(\hat \bz_{\bF},\kappa)\|_2 + {\|{\tilde \Delta_{{\hat \bS} \cap \bS^* }}\|_2}}{2\kappa_0^2-\kappa}
\end{align}\end{small}

According to the definition of $\theta$, it can be verified that $\theta(t,\kappa) = \max\{0,\frac{\lambda}{b} - {\kappa} |t - b| \}$ for $|t| > b$, and $\theta(0,\kappa) = \max\{0, \frac{\lambda}{b} - {\kappa} b\}$.  Therefore,
\begin{small}\begin{align}\label{eq:suboptimal-optimal-seg4}
&\|\theta(\hat \bz_{\bF},\kappa)\|_2
 = \big(\sum\limits_{j \in \bF \cap \hat \bS} (\max\{0,\frac{\lambda}{b} - {\kappa} |\hat \bz_j - b| \})^2 +
\sum\limits_{j \in \bF \setminus \hat \bS} (\max\{0, \frac{\lambda}{b} - {\kappa} b\})^2 \big)^{\frac{1}{2}}
\end{align}\end{small}
And it follows that
\begin{small}\begin{align}\label{eq:suboptimal-optimal-seg5}
&\|\Delta\|_2 \le \frac{1}{2\kappa_0^2-\kappa}\bigg(\big(\sum\limits_{j \in \bF \cap \hat \bS} (\max\{0,\frac{\lambda}{b} - {\kappa} |\hat \bz_j - b| \})^2 +
\sum\limits_{j \in \bF \setminus \hat \bS} (\max\{0, \frac{\lambda}{b} - {\kappa} b\})^2 \big)^{\frac{1}{2}} \bigg)
\end{align}\end{small}
This proves the result of this theorem.
\end{proof} 

\subsection{Proof of Theorem~\ref{theorem::optimal-rma}}
\begin{proof}
By the proof of Lemma~\ref{lemma::PGD-convergence}, we have
\begin{small}\begin{align*}
&\| 2{{{{\tilde \bD}}}^{\top}}{{{\tilde \bD}}}{{\tilde \bz}} - 2{{{{\tilde \bD}}}^{\top}}\bx  + {\dot \bR}(\tilde {\bz})\|_2  = 0
\end{align*}\end{small}
It follows that
\begin{small}\begin{align}\label{eq:optimal-rp-seg1}
&\| 2{{{\bD}}^{\top}}{{\bD}}{{\tilde \bz}} -2{{{\bD}}^{\top}}\bx  + {\dot \bR}(\tilde {\bz})\|_2  \nonumber \\
&= \| 2{{{\bD}}^{\top}}{{\bD}}{{\tilde \bz}} - 2{{{{\tilde \bD}}}^{\top}}{{{\tilde \bD}}}{{\tilde \bz}}
+2{{{{\tilde \bD}}}^{\top}}{{{\tilde \bD}}}{{\tilde \bz}} -2{{{\bD}}^{\top}}\bx
+ 2{{{{\tilde \bD}}}^{\top}}\bx - 2{{{{\tilde \bD}}}^{\top}}\bx
+ {\dot \bR}(\tilde {\bz})\|_2  \nonumber \\
& \le \| 2{{{\bD}}^{\top}}{{\bD}}{{\tilde \bz}} - 2{{{{\tilde \bD}}}^{\top}}{{{\tilde \bD}}}{{\tilde \bz}}\|_2 + \|2{{{\bD}}^{\top}}\bx - 2{{{{\tilde \bD}}}^{\top}}\bx\|_2
+ \| 2{{{{\tilde \bD}}}^{\top}}{{{\tilde \bD}}}{{\tilde \bz}} -2{{{{\tilde \bD}}}^{\top}}\bx + {\dot \bR}(\tilde {\bz})\|_2 \nonumber \\
&=\| 2{{{\bD}}^{\top}}{{\bD}}{{\tilde \bz}} - 2{{{{\tilde \bD}}}^{\top}}{{{\tilde \bD}}}{{\tilde \bz}}\|_2 + \|2{{{\bD}}^{\top}}\bx - 2{{{{\tilde \bD}}}^{\top}}\bx\|_2 \nonumber \\
&\le 2\|{{{\bD}}^{\top}} ({{\bD}}- {{{\tilde \bD}}}) {{\tilde \bz}}\|_2
+ 2\| ({{\bD}}-{{{\tilde \bD}}})^{\top} {{{\tilde \bD}}} {{\tilde \bz}}\|_2
+ 2\|{{{\bD}}^{\top}}\bx - {{{{\tilde \bD}}}^{\top}}\bx\|_2
\end{align}\end{small}
By ${\tilde L}({{\tilde \bz}}) \le {\tilde L}({\bz}^{(0)})$, we have $\|{{\tilde \bz}}\|_2 \le M_0$. By Lemma~\ref{lemma::D-approx}, with probability at least $1-6e^{-p}$, $\|\bD - {\tilde \bD}\|_2 \le C_{k,k_0}$. It follows from (\ref{eq:optimal-rp-seg1}) that
\begin{small}\begin{align*}
&\| 2{{{\bD}}^{\top}}{{\bD}}{{\tilde \bz}} -2{{{\bD}}^{\top}}\bx + {\dot \bR}(\tilde {\bz})\|_2 \nonumber \\
& \le 2\sigma_{\max}(\bD) C_{k,k_0}M_0 + 2C_{k,k_0} (\sigma_{\max}(\bD) + C_{k,k_0})M_0 + 2C_{k,k_0}\|\bx\|_2\nonumber \\
&= 2C_{k,k_0}M_0 (2\sigma_{\max}(\bD) + C_{k,k_0}) + 2C_{k,k_0}\|\bx\|_2
\end{align*}\end{small}
Also, by the proof of Lemma~\ref{lemma::PGD-convergence},
\begin{small}\begin{align*}
&\| 2{{{\bD}}^{\top}}{{\bD}} {{\bz}^*} -2{{{\bD}}^{\top}}\bx + {\dot \bR}({{\bz}^*})\|_2  = 0
\end{align*}\end{small}

Let $\Delta = {{\bz}^*} - {{\tilde \bz}}$, $\tilde \Delta = {\dot \bR}({\bz}^*) - {\dot \bR}(\tilde {\bz})$,
\begin{small}\begin{align*}
&\| 2{{{\bD}}^{\top}}{{\bD}} \Delta  + \tilde \Delta\|_2  \le 2C_{k,k_0}M_0 (2\sigma_{\max}(\bD) + C_{k,k_0}) + 2C_{k,k_0}\|\bx\|_2
\end{align*}\end{small}
Now following the proof of Theorem~\ref{theorem::suboptimal-optimal}, we have
\begin{small}\begin{align}\label{eq:optimal-rp-seg2}
&\|{\bz}^*-{\tilde \bz}\|_2 = \|\Delta\|_2 \nonumber \\
&\le \frac{1}{2\tau_0^2-\tau}\bigg(\big(\sum\limits_{j \in {\bG} \cap \tilde \bS} (\max\{0,\frac{\lambda}{b} - {\kappa} |{\tilde \bz}_j - b| \})^2 + \nonumber \\
&\sum\limits_{j \in {\bG} \setminus \tilde \bS} (\max\{0, \frac{\lambda}{b} - {\kappa} b\})^2 \big)^{\frac{1}{2}} + 2C_{k,k_0}M_0 (2\sigma_{\max}(\bD) + C_{k,k_0}) + 2C_{k,k_0}\|\bx\|_2 \bigg)
\end{align}\end{small}%
\end{proof}

\subsection{Proof of Theorem~\ref{theorem::optimal-rdr}}
We have the following lemma before proving Theorem~\ref{theorem::optimal-rdr}.

\begin{MyLemma}\label{lemma::matrix-prod-projection}
Suppose $\bT$ satisfies the $\ell^2$-norm preserving property in Definition~\ref{def:L2-norm-preseving}. If $m \ge 4c\log{\frac{4}{\delta}}$, then for any matrix $\bA \in \R^{p \times d}$, $\bB \in \R^{d \times q}$, with probability at least $1 - \delta$,
\begin{small}\begin{align}\label{eq:matrix-prod-projection}
&\|\bA \bT^{\top} \bT \bB - \bA \bB\| \le \|\bA\|_F \|\bB\|_F \sqrt{\frac{c}{m}\log{\frac{4}{\delta}}}
\end{align}\end{small}%
\end{MyLemma}
Lemma~\ref{lemma::matrix-prod-projection} can be proved using the definition of the $\ell^2$-norm preserving property in the same way that Lemma $6$ in \citep{Sarlos2006-large-matrix-random-projection} or \citep{Zhang2016-sparse-random-convex-concave} is proved.

\begin{proof}[Proof of Theorem~\ref{theorem::optimal-rdr}]
By the proof of Lemma~\ref{lemma::PGD-convergence}, we have
\begin{small}\begin{align*}
&\| 2{{{{\bar \bD}}}^{\top}}{{{\bar \bD}}}{{\bar \bz}} - 2{{{{\bar \bD}}}^{\top}}{\bar \bx}  + {\dot \bR}(\bar {\bz})\|_2  = 0
\end{align*}\end{small}
It follows that
\begin{small}\begin{align}\label{eq:optimal-rp-seg1}
&\| 2{{{\bD}}^{\top}}{{\bD}}{{\bar \bz}} -2{{{\bD}}^{\top}}\bx  + {\dot \bR}(\bar {\bz})\|_2  \nonumber \\
&= \| 2{{{\bD}}^{\top}}{{\bD}}{{\bar \bz}} - 2{{{{\bar \bD}}}^{\top}}{{{\bar \bD}}}{{\bar \bz}}
+2{{{{\bar \bD}}}^{\top}}{{{\bar \bD}}}{{\bar \bz}} -2{{{\bD}}^{\top}}\bx
+ 2{{{{\bar \bD}}}^{\top}}{\bar\bx} - 2{{{{\bar \bD}}}^{\top}}{\bar\bx}
+ {\dot \bR}(\bar {\bz})\|_2  \nonumber \\
& \le \| 2{{{\bD}}^{\top}}{{\bD}}{{\bar \bz}} - 2{{{{\bar \bD}}}^{\top}}{{{\bar \bD}}}{{\bar \bz}}\|_2 + \|2{{{\bD}}^{\top}}\bx - 2{{{{\bar \bD}}}^{\top}}{\bar\bx}\|_2
+ \| 2{{{{\bar \bD}}}^{\top}}{{{\bar \bD}}}{{\bar \bz}} -2{{{{\bar \bD}}}^{\top}}{\bar\bx} + {\dot \bR}(\bar {\bz})\|_2 \nonumber \\
& = \| 2{{{\bD}}^{\top}}{{\bD}}{{\bar \bz}} - 2{{{{\bar \bD}}}^{\top}}{{{\bar \bD}}}{{\bar \bz}}\|_2 + \|2{{{\bD}}^{\top}}\bx - 2{{{{\bar \bD}}}^{\top}}{\bar\bx}\|_2 \nonumber \\
& = 2 \|{{{\bD}}^{\top}}(\bI - \bT^{\top}{\bT}){{\bD}} {\bar \bz} \|_2 + 2 \|{{{\bD}}^{\top}}(\bI - \bT^{\top}{\bT}) \bx\|_2
\end{align}\end{small}%

By ${\bar L}({{\bar \bz}}) \le {\bar L}({\bz}^{(0)})$, we have $\|{{\bar \bz}}\|_2 \le M_1$. According to Lemma~\ref{lemma::matrix-prod-projection}, with probability at least $1 - \delta$,
\begin{small}\begin{align}\label{eq:optimal-rp-seg2}
&2 \|{{{\bD}}^{\top}}(\bI - \bT^{\top}{\bT}){{\bD}} {\bar \bz} \|_2 + 2 \|{{{\bD}}^{\top}}(\bI - \bT^{\top}{\bT}) \bx\|_2 \nonumber \\
&\le 2\|\bD\|_F\sigma_{\max}(\bD)M_1\sqrt{\frac{c}{m}\log{\frac{4}{\delta}}} + 2\|\bD\|_FM_1\sqrt{\frac{c}{m}\log{\frac{4}{\delta}}} \nonumber \\
&\le 2\|\bD\|_F{M_1}\sqrt{\frac{c}{m}\log{\frac{4}{\delta}}} (\sigma_{\max}(\bD)+1)
\end{align}\end{small}
Combining (\ref{eq:optimal-rp-seg1}) and (\ref{eq:optimal-rp-seg2}), we have
\begin{small}\begin{align*}
&\| 2{{{\bD}}^{\top}}{{\bD}}{{\bar \bz}} -2{{{\bD}}^{\top}}\bx + {\dot \bR}(\bar {\bz})\|_2 \le 2\|\bD\|_F{M_1}\sqrt{\frac{c}{m}\log{\frac{4}{\delta}}} (\sigma_{\max}(\bD)+1)
\end{align*}\end{small}
Also, by the proof of Lemma~\ref{lemma::PGD-convergence},
\begin{small}\begin{align*}
&\| 2{{{\bD}}^{\top}}{{\bD}} {{\bz}^*} -2{{{\bD}}^{\top}}\bx + {\dot \bR}({{\bz}^*})\|_2  = 0
\end{align*}\end{small}
Let $\Delta = {{\bz}^*} - {{\bar \bz}}$, $\bar \Delta = {\dot \bR}({\bz}^*) - {\dot \bR}(\bar {\bz})$,
\begin{small}\begin{align*}
&\| 2{{{\bD}}^{\top}}{{\bD}} \Delta  + \bar \Delta\|_2  \le 2\|\bD\|_F{M_1}\sqrt{\frac{c}{m}\log{\frac{4}{\delta}}} (\sigma_{\max}(\bD)+1)
\end{align*}\end{small}
Now following the proof of Theorem~\ref{theorem::suboptimal-optimal}, with probability at least $1 - \delta$,
\begin{small}\begin{align}\label{eq:optimal-rp-seg3}
&\|{\bz}^*-{\bar \bz}\|_2 = \|\Delta\|_2 \nonumber \\
&\le \frac{1}{2\eta_0^2-\eta}\bigg(\big(\sum\limits_{j \in {\bH} \cap \bar \bS} (\max\{0,\frac{\lambda}{b} - {\kappa} |{\bar \bz}_j - b| \})^2 + \nonumber \\
&\sum\limits_{j \in {\bH} \setminus \bar \bS} (\max\{0, \frac{\lambda}{b} - {\kappa} b\})^2 \big)^{\frac{1}{2}} + 2\|\bD\|_F{M_1}\sqrt{\frac{c}{m}\log{\frac{4}{\delta}}} (\sigma_{\max}(\bD)+1) \bigg)
\end{align}\end{small}
\end{proof}


\end{document}